\documentclass[12pt]{amsart}
\usepackage{tikz, hyperref, amssymb,enumerate}
\usepackage[margin=2cm]{geometry}

\newtheorem{thm}{Theorem}[section]
\newtheorem{cor}[thm]{Corollary}

\newtheorem{prop}[thm]{Proposition}

\theoremstyle{definition}
\newtheorem{dfn}[thm]{Definition}

\newtheorem{ntn}[thm]{Notation}

\theoremstyle{remark}
\newtheorem{rmk}[thm]{Remark}

\newtheorem{example}[thm]{Example}

\newcommand{\FF}{\mathbb{F}}

\newcommand{\NN}{\mathbb{N}}
\newcommand{\QQ}{\mathbb{Q}}
\newcommand{\TT}{\mathbb{T}}
\newcommand{\ZZ}{\mathbb{Z}}

\newcommand{\Mm}{\mathcal{M}}

\newcommand{\Oo}{\mathcal{O}}

\newcommand{\1}{\mathbf{1}}

\newcommand{\id}{\operatorname{id}}

\newcommand{\image}{\operatorname{Im}}
\newcommand{\Aut}{\operatorname{Aut}}

\newcommand{\coker}{\operatorname{coker}}
\newcommand{\Hom}{\operatorname{Hom}}
\newcommand{\Ext}{\operatorname{Ext}}

\newcommand{\Obj}{\operatorname{Obj}}

\newcommand{\dom}{\operatorname{dom}}
\newcommand{\cod}{\operatorname{cod}}

\newcommand{\Bcub}[1]{B^{#1}}
\newcommand{\Ccub}[1]{C^{#1}}
\newcommand{\Hcub}[1]{H^{#1}}
\newcommand{\Zcub}[1]{Z^{#1}}
\newcommand{\dcub}[1]{\delta^{#1}}

\title[Crossed products]{Crossed products and twisted $k$-graph algebras}

\author{Nathan Brownlowe}

\author{Valentin Deaconu} 

\author{Alex Kumjian}
\address{Valentin Deaconu, Alex Kumjian\\ Department of Mathematics (084)\\ University
of Nevada\\ Reno NV 89557-0084\\ USA} \email{vdeaconu@unr.edu, alex@unr.edu}

\author{David Pask}
\address{Nathan Brownlowe, David Pask \\ School of Mathematics and
Applied Statistics  \\
The University of Wollongong\\
NSW  2522\\
AUSTRALIA} \email{nathanb@uow.edu.au, dpask@uow.edu.au}

\thanks{This research was supported by the Australian Research Council and the University of Wollongong Research Committee.}

\subjclass[2010]{Primary 46L05; Secondary 18G60, 55N10.}

\keywords{Higher-rank graph; $C^*$-algebra; cohomology; crossed product.}

\date{\today}

\begin{document}

\begin{abstract}
An automorphism $\beta$ of a $k$-graph $\Lambda$ induces a crossed product $C^* ( \Lambda ) \rtimes_\beta \ZZ$
which is isomorphic to a $(k+1)$-graph algebra $C^* ( \Lambda \times_\beta \ZZ )$. In this paper we show how this process interacts with $k$-graph $C^*$-algebras which have been twisted by an element of their second cohomology group. This analysis is done using a long exact sequence in cohomology associated to this data. We conclude with some examples. 

\end{abstract}

\maketitle

\noindent
A higher-rank graph (or $k$-graph) is a countable category $\Lambda$ together with a functor $d : \Lambda \to \NN^k$ satisfying a factorisation property. For $k=1$,  $\Lambda$ is the path category of a directed graph $E_\Lambda$. In general we view a $k$-graph   as a higher dimensional analog of a directed graph. In \cite{KP2000} it was shown how to associate a $C^*$-algebra to a $k$-graph in such a way that for $k=1$ we have $C^* ( \Lambda ) =C^* ( E_\Lambda )$. 

The universal property of a $k$-graph $C^*$-algebra $C^* ( \Lambda )$ implies that an automorphism $\beta$ of $\Lambda$ induces an automorphism $\beta$ of $C^* ( \Lambda )$ and hence gives rise to a crossed product $C^* ( \Lambda ) \rtimes_\beta \ZZ$. 
The results of \cite{FPS2009} show that there is a $(k+1)$-graph $\Lambda \times_\beta \ZZ$ such that $C^* ( \Lambda ) \rtimes_\beta \ZZ$ is isomorphic to $C^* ( \Lambda \times_\beta \ZZ )$. 
The purpose of this paper is to examine how this situation generalizes in the setting of twisted $k$-graph $C^*$-algebras.

Recent attention has been drawn to the homological properties of a $k$-graph $\Lambda$ which are nontrivial when $k \ge 2$. 
Specifically in \cite{KPS3,KPS4} two cohomology theories for a $k$-graph $\Lambda$ are described: cubical and categorical. 
Twisted versions of $k$-graph $C^*$-algebras are introduced in both cases using $2$-cocycles.
Here we work with the cubical cohomology which is more tractable. 
If $\varphi$ is a (cubical) $\TT$-valued $2$-cocycle on  $\Lambda$, the twisted  $C^*$-algebra is 
denoted by $C^*_\varphi(\Lambda)$.

In \cite{KPS4} it is shown that the cubical and categorical cohomologies for a $k$-graph agree for $H^0 , H^1$ and $H^2$.  
An isomorphism between the two twisted versions of $k$-graph $C^*$-algebras compatible with the isomorphism in $H^2$ was also proven in \cite{KPS4}.

If $\beta$ is an automorphism of a $k$-graph $\Lambda$, then \cite{KPS3} gives a long exact sequence for the homology of the $(k+1)$-graph $\Lambda \times_\beta \ZZ$ in the categorical context. In this paper we describe the analogous cohomology sequence in the cubical context (see Proposition \ref{prop:les}) and use it to generalise the result in \cite{FPS2009} in three different ways.

Our main result, Theorem~\ref{thm:main} shows that if we twist the $C^*$-algebra of $\Lambda \times_\beta \ZZ$  by a 2-cocycle $\varphi$ then the resulting $C^*$-algebra $C^*_\varphi(\Lambda \times_\beta \ZZ)$ is isomorphic to the crossed product of a certain twisted $C^*$-algebra of $\Lambda$ (with twisting cocycle obtained by restricting $\varphi$) by an automorphism associated to $\beta$ and $\varphi$. 
Applying this result in different contexts, associated to the exact sequence outlined in 
Proposition~\ref{prop:les} yields Corollary~\ref{cor:qfa} 
which deals with the case where the class of the restriction of $\varphi$ is trivial; 
Corollary~\ref{cor:2} asserts that if $\psi$ is a $2$-cocycle on $\Lambda$ 
whose cohomology class is left invariant by $\beta$, then there is an automorphism of 
$C^*_\psi(\Lambda)$ which is compatible with $\beta$ for which the crossed product is
isomorphic to a twisted $C^*$-algebra of $\Lambda \times_\beta \ZZ$.
The case when $\beta$ is trivial which motivated this work is discussed in Corollary~\ref{cor:3}.

We conclude with a section of examples of twisted $k$-graph $C^*$-algebras arising as crossed products. In each case the twisted $k$-graph $C^*$-algebra lies in a classifiable class of $C^*$-algebras. 
In Example~\ref{ex:1} we consider quasifree automorphisms on Cuntz algebras and show how they arise in the setting 
of Corollary~\ref{cor:qfa}. 
In Example~\ref{ex:2} we use Theorem~\ref{thm:main} to compute the cohomology of a $2$-graph with infinitely many vertices which arises as a crossed product.
In Example~\ref{ex:3} we use other techiniques to compute the cohomology of a $3$-graph with one vertex which arises as a crossed product.  
We also consider a family of $2$-cocycles on the $3$-graph for which the associated $C^*$-algebra is isomorphic to $\mathcal{O}_2$.

\section{Background}
We start this section by giving some background on $k$-graphs, their cubical 
cohomology, 
and crossed-product graphs induced by automorphisms of $k$-graphs. We then 
prove the existence of a long exact sequence involving the cohomology groups 
of a 
$k$-graph and a crossed product graph. We finish with recalling the twisted 
$k$-graph $C^*$-algebras introduced in \cite{KPS3}.

\subsection{Higher-rank graphs}

We adopt the conventions of \cite{KP2000, kps2, pqr1} for $k$-graphs. Given a nonnegative integer $k$, a
\emph{$k$-graph} is a nonempty countable small category $\Lambda$ equipped with a functor $d
:\Lambda \to \NN^k$ satisfying the \emph{factorisation property}: for all $\lambda \in \Lambda$ and
$m,n \in \NN^k$ such that $d( \lambda )=m+n$ there exist unique $\mu ,\nu \in \Lambda$ such that
$d(\mu)=m$, $d(\nu)=n$, and $\lambda=\mu \nu$. When $d(\lambda )=n$ we say $\lambda$ has
\emph{degree} $n$. We will typically use $d$ to denote the degree functor in any $k$-graph in this
paper.

For $k \ge 1$, the standard generators of $\NN^k$ are denoted $e_1, \dots, e_k$, and for $n \in
\NN^k$ and $1 \le i \le k$ we write $n_i$ for the $i^{\rm th}$ coordinate of $n$. For $n = (n_1 ,
\ldots , n_k ) \in \NN^k$ let $| n | := \sum_{i=1}^k n_i$;  for $\lambda \in \Lambda$ we define $|
\lambda | := | d ( \lambda ) |$. For $m,n \in \NN^k$, we write $m \vee n$ for the coordinatewise
maximum of the two, and write $m \le n$ if $m_i \le n_i$ for $i=1 , \ldots , k$.

For $n \in \NN^k$, we write $\Lambda^n$ for $d^{-1}(n)$. The \emph{vertices} of $\Lambda$ are the
elements of $\Lambda^0$. The factorisation property implies that $o \mapsto \id_o$ is a bijection
from the objects of $\Lambda$ to $\Lambda^0$. We will frequently and without further comment use
this bijection to identify $\Obj(\Lambda)$ with $\Lambda^0$. The domain and codomain maps in the
category $\Lambda$ then become maps $s,r : \Lambda \to \Lambda^0$. More precisely, for $\alpha
\in\Lambda$, the \emph{source} $s(\alpha)$ is the identity morphism associated with the object
$\dom(\alpha)$ and similarly, $r(\alpha) = \id_{\cod(\alpha)}$. An \textit{edge} is a morphism $f$
with $d(f) = e_i$ for some $i=1, \ldots , k$.

Let $\lambda$ be an element of a $k$-graph $\Lambda$ and suppose $m,n \in \NN^k$ satisfy $0 \le m
\le n \le d(\lambda)$. By the factorisation property there exist unique elements $\alpha, \beta,
\gamma \in \Lambda$ such that
\[
\lambda = \alpha\beta\gamma, \quad d(\alpha) = m,
\quad d(\beta) = n - m,\quad\text{and}\quad d(\gamma) = d(\lambda) - n.
\]
We define $\lambda(m,n) := \beta$. In particular $\alpha = \lambda(0, m)$ and $\gamma = \lambda(n,
d(\lambda))$.

For $\alpha,\beta\in\Lambda$ and $E \subset \Lambda$, we write $\alpha E$ for $\{\alpha\lambda :
\lambda \in E, r(\lambda) = s(\alpha)\}$ and $E\beta$ for $\{\lambda\beta : \lambda \in E,
s(\lambda) = r(\beta)\}$. So for $u,v \in \Lambda^0$, we have $uE = E \cap r^{-1}(u)$, $E v = E
\cap s^{-1}(v)$ and $uEv = uE \cap Ev$.


Suppose that  $\Lambda$ is \emph{row-finite with no sources}, that is, for all $v \in \Lambda^0$ and $n \in \NN^k$ we have $0 < | v \Lambda^n | < \infty$.  
By \cite[Remark A.3]{LS},  $\Lambda$ is \emph{cofinal} if for all $v , w \in \Lambda^0$ there is $N \in \NN^k$ such that for all $\alpha \in v \Lambda^N$ we have $w \Lambda s (\alpha) \neq \emptyset$.
And by \cite[Lemma 3.2 (iv)]{RoSi}, $\Lambda$ is \textit{aperiodic} (or satisfies the \textit{aperiodicity condition}) if for every $v\in{\Lambda}^0$ and each pair $m \ne n\in\mathbb{N}^k$, 
there is $\lambda\in v\Lambda$ such that $d(\lambda)\ge m\vee n$ and
\begin{equation} \label{eq:lp}
\lambda(m,m+d(\lambda)-(m\vee n))\ne \lambda(n,n+d(\lambda)-(m\vee n)).
\end{equation}

A $k$-graph $\Lambda$ can be visualized by its \emph{$1$-skeleton}: This is a directed graph $E_\Lambda$ with vertices $\Lambda^0$ and edges $\cup_{i=1}^k \Lambda^{e_i}$ which have range and source in $E_\Lambda$ determined by their range and source in $\Lambda$.
Each edge in $E_\Lambda$ with degree $e_i$ is assigned the same colour, so $E_\Lambda$ is a coloured graph. It is common to call edges with degree $e_1$ in $\Lambda$ blue edges in $E_\Lambda$ and draw them with solid lines; edges with degree $e_2$ in $\Lambda$ are then called red edges and are drawn as dashed lines. 
In practice, along with the $1$-skeleton we give a collection of commuting squares or factorisation rules which relate the edges of $E_\Lambda$ that occur in the factorisation of morphisms of degree
$e_i+e_j$ ($i \neq j$) in $\Lambda$. 
For more information about $1$-skeletons we refer the reader to \cite{RSY2003}.

A functor $\beta : \Lambda \to \Gamma$ between $k$-graphs is a \emph{$k$-graph morphism} if it preserves degree, that is $d_\Gamma \circ \beta = d_\Lambda$. 
If $\Gamma = \Lambda$ and $\beta$ is invertible then $\beta$ is an \emph{automorphism}. 
The collection $\operatorname{Aut} \Lambda$ of automorphisms of $\Lambda$ forms a group under composition.


Let $T_1$ be the category $\NN$ regarded as a $1$-graph with degree functor given by the identity map. 

\subsection{Cubical cohomology of $k$-graphs}\label{ssec:cub coh}

For $k \ge 0$ define $\mathbf{1}_k := \sum^k_{i=1} e_i \in \NN^k$. 

\begin{dfn}
Let $\Lambda$ be a $k$-graph. For $r \ge 0$ let
$Q_r ( \Lambda ) = \{ \lambda \in \Lambda : d ( \lambda ) \le \mathbf{1}_k , |\lambda| = r \}$. 
\end{dfn}

\noindent
We have $Q_0(\Lambda) = \Lambda^0$, $Q_1 ( \Lambda ) = \bigcup_{i=1}^k \Lambda^{e_i}$ the set of edges in $\Lambda$ and $Q_r(\Lambda) = \emptyset$ if $r > k$. For $0 < r \le k$
the set $Q_r(\Lambda)$ consists of the morphisms in $\Lambda$ which may be expressed as the composition of a sequence of $r$ edges
with distinct degrees. We regard elements of $Q_r(\Lambda)$ as unit $r$-cubes in the sense that each one gives rise to a
commuting diagram of edges in $\Lambda$ shaped like an $r$-cube. In particular, when $r \ge 1$, each element of
$Q_r(\Lambda)$ has $2r$ faces in $Q_{r-1}(\Lambda)$ defined as follows.

\begin{dfn}
Fix $\lambda \in Q_r(\Lambda)$ and write $d( \lambda ) = e_{i_1} + \cdots + e_{i_r}$ where $i_1 < \cdots < i_r$. For $1
\le j \le r$, define $F_j^0(\lambda)$ and $F_j^1(\lambda)$ to be the unique elements of $Q_{r-1} ( \Lambda )$ such that
there exist $\mu,\nu \in \Lambda^{e_{i_j}}$ satisfying
\[
F^0_j(\lambda)\nu = \lambda = \mu F^1_j(\lambda).
\]
\end{dfn}

\noindent
In \cite{KPS3} the cubical homology of $\Lambda$ is identified with the homology of the complex 
$(\ZZ Q_*, \partial_*)$ where the boundary map $\partial_r : \ZZ Q_r \to \ZZ Q_{r-1}$ is determined by
\[
\partial_r \lambda =
\sum^{r}_{j=1}\sum_{\ell=0}^1 (-1)^{j+\ell} F_j^\ell(\lambda).
\]
\begin{rmk}
If $\beta \in \Aut ( \Lambda )$, then it is straightforward to check that the induced action of $\beta$ on $\ZZ Q_r$ commutes with $\partial_r$.
We first observe that
\[
F^0_j ( \beta \lambda ) = \beta  F^0_j ( \lambda ) \text{ and }
F^1_j ( \beta \lambda ) = \beta  F^1_j ( \lambda );
\]
\noindent and hence
\[
\partial_r (\beta \lambda) =
\sum^{r}_{j=1}\sum_{\ell=0}^1 (-1)^{j+\ell} F_j^\ell( \beta \lambda) = 
\sum^{r}_{j=1}\sum_{\ell=0}^1 (-1)^{j+\ell} \beta F_j^\ell(\lambda)=\beta \partial_r (\lambda) .
\]
\end{rmk}

\begin{ntn}
Let $\Lambda$ be a $k$-graph and let $A$ be an abelian group. For $r \ge 0$, we write
$\Ccub{r}(\Lambda, A)$ for the collection of all functions $f : Q_r(\Lambda) \to A$. 
Identify $\Ccub{r}(\Lambda, A)$ with $\Hom(\ZZ Q_r (\Lambda), A)$ in the usual way. 
Define maps $\dcub{r} : \Ccub{r}(\Lambda,A) \to \Ccub{r+1}(\Lambda,A)$ by
\[
\dcub{r}(f)(\lambda) := f(\partial_{r+1}(\lambda)) =
\sum^{r+1}_{j=1}\sum_{\ell=0}^1 (-1)^{j+\ell}f(F_j^\ell(\lambda)).
\]
Then $( \Ccub{*}(\Lambda, A), \dcub{*})$ is a cochain complex.
\end{ntn}

\begin{dfn}
We define the \emph{cubical cohomology} $\Hcub*(\Lambda, A)$ of the $k$-graph $\Lambda$ with coefficients
in $A$ to be the cohomology of the complex $(\Ccub*(\Lambda,A),\delta^*)$; that is $\Hcub{r}(\Lambda,A) :=
\ker(\dcub{r})/\image(\dcub{r-1})$. For $r \ge 0$, we write $\Zcub{r}(\Lambda,A) := \ker(\dcub{r})$
for the group of $r$-cocycles, and for $r > 0$, we write $\Bcub{r}(\Lambda,A) = \image(\dcub{r-1})$
for the group of $r$-coboundaries.
\end{dfn}

\begin{rmk}
For each $0 \le r \le k$ we define $\beta^* : C^r ( \Lambda , A ) \to C^r (  
\Lambda , A )$ by $\beta(f)=f\circ\beta$. For each $f \in C^r ( \Lambda , A )$ 
and $\lambda\in Q_r(\Lambda)$ we have
\[
\delta^r \beta^*(f) ( \lambda ) 
= \sum_{j=1}^{r+1}\sum_{\ell=0}^1 f(F_j^\ell(\beta(\lambda)))=\delta^r  
(f)(\beta\lambda)=\beta^*\delta^r(f)( \lambda ),
\]
and so $\beta^* \circ \delta^r = \delta^r \circ \beta^*$. Hence 
$\beta^*$ 
induces a homomorphism $\beta^* : 
H^* ( \Lambda  , A) \to H^* ( \Lambda , A )$.
\end{rmk}

\subsection{Crossed product graphs}\label{ssec:crossed}

Recall from
\cite{FPS2009} that if $\Lambda$ is a row-finite $k$-graph with no sources and $\beta \in \Aut  \Lambda$, then there is a
$(k+1)$-graph $\Lambda \times_\beta \ZZ$\label{pg:cpgraph} with morphisms $\Lambda \times \NN$, range and source
maps given by $r(\lambda,n) = (r(\lambda),0)$, $s(\lambda,n) = (\beta^{-n}(s(\lambda)), 0)$, degree map given by
$d(\lambda,n) = (d(\lambda),n)$ and composition given by $(\lambda,m)(\mu,n) := (\lambda\beta^m(\mu), m+n)$. Evidently, $\Lambda \times_\beta \ZZ$ is also row-finite with no sources and $( \Lambda \times_\beta \ZZ )^0 = \Lambda^0 \times \{ 0 \}$.

\begin{rmk} \label{rmk:betatriv}
If $\beta=\id$, note that $\Lambda\times_\beta\ZZ=\Lambda\times T_1$.
\end{rmk}

Recall from \cite[\S 4]{KPS3} that we may describe the $r$-cubes of $\Lambda \times_\beta \ZZ$ in terms of the cubes of $\Lambda$.
The $0$-cubes are given 
by $Q_0 ( \Lambda \times_\beta \ZZ ) = Q_0 ( \Lambda ) \times \{ 0 \}$. For 
each $0 \le r \le k$ the $(r+1)$-cubes are given by
\[
Q_{r+1}(\Lambda\times_\beta\ZZ)=\{(\lambda,1):\lambda\in 
Q_r(\Lambda)\}\cup\{(\lambda,0):\lambda\in Q_{r+1}(\Lambda)\}.
\]
Observe that for $\lambda \in Q_{r} ( \Lambda )$,  $F_j^\ell(\lambda , 0 ) = (F_j^\ell(\lambda), 0)$ and
\begin{equation}\label{eq: faces in cross graph}
F_j^\ell(\lambda , 1) =
\begin{cases}
(F_j^\ell(\lambda), 1) & \text{if } j \le r, \\
(\lambda, 0) &  \text{if } j = r+1, \ell = 0,\\
(\beta^{-1}(\lambda), 0) &  \text{if } j  = r+1, \ell = 1.
\end{cases}
\end{equation}
So for $f \in \Ccub{r}(\Lambda \times_\beta \ZZ, A)$, we have
\[
\dcub{r}(f)(\lambda, 0) =
 \sum^{r+1}_{j=1}\sum_{\ell=0}^1 (-1)^{j+\ell} f(F_j^\ell(\lambda), 0) \quad \text{for } \lambda \in Q_{r+1} ( \Lambda ) 
 \]
 and
 \begin{equation}\label{eq: delta r of f of lambda one}
\dcub{r}(f)(\lambda, 1) =
(-1)^{r+1}(f(\lambda, 0) - f(\beta^{-1}(\lambda), 0)) + \sum^{r}_{j=1}\sum_{\ell=0}^1 (-1)^{j+\ell} f(F_j^\ell(\lambda), 1), 
\end{equation}
for each $\lambda \in Q_{r} ( \Lambda )$.

\subsection{The long exact sequence of cohomology}

Suppose $\beta$ is an automorphism of a $k$-graph $\Lambda$. In 
\cite[Theorem~4.13]{KPS3} the authors presented a long exact sequence relating 
the homology groups of $\Lambda$ and $\Lambda\times_\beta\ZZ$. In the next 
result we present the corresponding long exact sequence of cohomology. 

\begin{prop} \label{prop:les}
Suppose $\beta$ is an automorphism of a $k$-graph $\Lambda$, and $A$ is an 
abelian group. There is a long exact sequence
\[
\begin{split}
    0 \longrightarrow &H^0 ( \Lambda \times_\beta \ZZ , A ) 
    \stackrel{i^*}{\longrightarrow}
        H^0 ( \Lambda  , A ) \xrightarrow{1-\beta^*} H^0 ( \Lambda , A ) 
        \stackrel{j^*}{\longrightarrow}
        H^1 (\Lambda \times_\beta \ZZ , A ) \stackrel{i^*}{\longrightarrow} 
        \cdots  \\
 & \cdots \stackrel{1-\beta^*}{\longrightarrow} H^r ( \Lambda , A ) 
 \stackrel{j^*}{\longrightarrow} H^{r+1} (\Lambda \times_\beta \ZZ , A )
        \xrightarrow{i^*} H^{r+1} (\Lambda , A ) 
        \stackrel{1-\beta^*}{\longrightarrow}
        H^{r+1} (\Lambda, A ) \stackrel{j^*}{\longrightarrow}  \cdots \\
        & \stackrel{j^*}{\longrightarrow} H^k ( \Lambda \times_\beta \ZZ , A ) 
        \stackrel{i^*}{\longrightarrow} H^k (\Lambda , A )
        \xrightarrow{1-\beta^*} H^k (\Lambda , A ) 
        \stackrel{j^*}{\longrightarrow}
        H^{k+1} (\Lambda \times_\beta \ZZ , A ) \longrightarrow 0,
\end{split}
\]
where 
\[
i^*(f)(\lambda)=f(\lambda,0)
\] 
for each $f\in Z^r(\Lambda\times_\beta\ZZ,A)$, and 
\[
j^*(f)(\lambda,0)=0\quad\text{ and }\quad j^*(f)(\lambda,1)=f(\lambda)
\]
for each $f\in Z^{r-1}(\Lambda,A)$.
\end{prop}

\begin{proof}
For each $0\le r\le k$ the maps $i: \ZZ Q_r (\Lambda)\to \ZZ Q_r (\Lambda\times_\beta\ZZ)$ and
$j: \ZZ Q_{r+1}(\Lambda\times_\beta\ZZ)\to \ZZ Q_r (\Lambda)$ determined by 
\begin{align*}
&& i(\lambda) &=(\lambda, 0)  && \text{for } \lambda \in Q_r (\Lambda) &&\\
&& j(\lambda,\ell)& =
\begin{cases}
0 & \text{if } \ell = 0 \\
\lambda & \text{if } \ell = 1
\end{cases} 
&& \text{for } (\lambda,\ell) \in Q_{r+1}(\Lambda\times_\beta\ZZ) &&
\end{align*}
induce maps $i^*: C^r(\Lambda\times_\beta\ZZ, A) \to C^r(\Lambda, A)$ and
$j^*: C^{r}(\Lambda, A) \to C^{r+1}(\Lambda\times_\beta\ZZ, A)$ given by
\begin{align*}
&& i^*(f)(\lambda) &=f(\lambda,0)  && \text{for } f \in C^r(\Lambda\times_\beta\ZZ, A) &&\\
&& j^*(f)(\lambda,\ell) &=
\begin{cases}
0&\text{if }\ell=0\\
f(\lambda)&\text{if }\ell=1
\end{cases}&& \text{for }f \in C^{r}(\Lambda, A). &&
\end{align*}
%
\noindent
Using the description of the cubes in $\Lambda\times_\beta\ZZ$ we obtain a 
short exact sequence of complexes $E$ where
\begin{equation}\label{eq: split short exact}
E_r:  0\to C^{r-1}(\Lambda, A)\stackrel{j^*}{\to} C^r(\Lambda\times_\beta\ZZ , A )\stackrel{i^*}{\to} C^r(\Lambda, A)\to 0,
\end{equation}
\noindent
since $i^*$ and $j^*$ commute with the coboundary maps. Indeed, 
\[
( \delta^r i^* ) f ( \lambda ) = \sum_{j=1,\ell=0}^{r+1,1} ( i^* f ) ( F_j^\ell ( \lambda ) ) 
= \sum_{j=1,\ell=0}^{r+1,1} f  ( F_j^\ell ( \lambda ) , 0 ) = \delta^r f ( \lambda , 0 ) =  ( i^*  \delta^r ) f ( \lambda )
\]

\noindent
for $f\in C^r(\Lambda\times_\beta\ZZ,A)$ and 
\begin{align*}
( \delta^{r}  j^* ) f ( \lambda , 1 )  &= (-1)^{r+1} \big( (j^* f ) ( \lambda , 0 ) - ( j^* f) ( \beta^{-1} ( \lambda ) , 0 ) \big) + \sum_{j=1,\ell=0}^{r,1} (-1)^{j+\ell} ( j^* f ) ( F^\ell_j ( \lambda ) , 1 ) \\
&= \sum_{j=1,\ell=0}^{r,1} (-1)^{j+\ell}  f  ( F^\ell_j ( \lambda ) )  = \delta^{r-1} f ( \lambda )  = (j^* \delta^{r-1} ) f ( \lambda , 1) , \text{ and  }
\end{align*}
\[
( \delta^{r}  j^* ) f ( \lambda , 0 )  = \sum_{j=1,\ell=0}^{r+1,1} (-1)^{j+\ell} ( j^* f) ( F^\ell_j ( \lambda ) , 0 ) \\
= 0 =  (j^* \delta^{r-1} ) f ( \lambda , 0 ) 
\]

\noindent
for $f\in C^{r-1}(\Lambda,A)$.

Using the long exact sequence associated to a short exact sequence of homology complexes 
(see \cite[Theorem II.4.1]{ML}) applied to a short exact sequence of cohomology complexes
with the appropriate reindexing, 
we obtain the long exact sequence
\[
\dots \stackrel{\delta^r_E}{\to} H^r(\Lambda,A) \stackrel{j^*}{\to} 
H^{r+1}(\Lambda\times_\beta\ZZ,A)\stackrel{i^*}{\to}H^{r+1}(\Lambda,A)
\stackrel{\delta^{r+1}_E}{\to}H^{r+1}(\Lambda,A)\to\dots.
\]  
Indeed, the boundary map $\delta_E$  (see \cite[II.4]{ML}) is defined as follows. Start with $m\in Z^r(\Lambda, A)$ and take the lift $n\in C^r(\Lambda\times_\beta\ZZ, A)$ given by $n(\lambda,0)=m(\lambda)$ and $n(\lambda,1)=0$.
Then $\delta^rn(\lambda,0)=0$ for all $\lambda$, so there is $c\in Z^r(\Lambda, A)$ such that $\delta^rn=j^*c$, that is $\delta^rn(\lambda, 1)=c(\lambda)$. 
Then $\delta_E$ takes the class of $m$ into the class of $c$. Using (\ref{eq: faces in cross graph}) and 
(\ref{eq: delta r of f of lambda one}) we get
\[
c(\lambda)=j^*(c)(\lambda,1)=\delta^rn(\lambda,1)
=(-1)^{r+1}(n(\lambda,0)-n(\beta^{-1}(\lambda),0))
=(-1)^{r+1}(m-m\circ\beta^{-1})(\lambda).
\]
Hence we see that $\delta^r_E=(-1)^{r+1}(1-(\beta^{-1})^*)$. By using a 
similar argument as in \cite[Theorem 4.13]{KPS3}, the sequence remains exact 
after replacing $\delta_E$ with $1-\beta^*$.
\end{proof}

\noindent
If $\beta = \id$, then as noted in Remark~\ref{rmk:betatriv} $\Lambda \times_\beta \ZZ$ may be identified with  $\Lambda \times T_1$
and $1-\beta^* = 0$.  Hence, for all $r$, we have the short exact sequence
\[
0 \rightarrow  H^r(\Lambda, A) \stackrel{j^*}{\longrightarrow} H^{r+1} (\Lambda \times T_1, A )
        \xrightarrow{i^*} H^{r+1} (\Lambda , A) \rightarrow 0.
\]
Moreover there is a map $\sigma: C^r(\Lambda, A)\to C^r(\Lambda \times T_1, A)$
such that $\sigma(f)(\lambda, 0)= f(\lambda)$ for $\lambda\in Q_{r}(\Lambda)$ and, if $r \ge 1$,
$\sigma(f)(\lambda, 1)= 0$ for $\lambda\in Q_{r-1}(\Lambda)$.  
It is straightforward to check that $\sigma$ intertwines boundary maps and that 
$i^*\sigma(f) = f$ for all $f \in C^r(\Lambda, A)$.   Hence, the map 
\begin{equation}\label{eq: Xi}
\Xi: (f, g) \in Z^{r}(\Lambda, A) \oplus Z^{r+1}(\Lambda, A) \mapsto 
j^*(f) + \sigma(g) \in Z^{r+1}(\Lambda \times T_1, A)
\end{equation}
is an isomorphism which intertwines the boundary maps.  We thereby obtain the following result:

\begin{cor}\label{cor:betatriv}
If $\beta = \id$, then $H^0(\Lambda \times T_1, A) \cong H^0(\Lambda, A)$ and 
for $r \ge 0$ the map on cohomology induced by $\Xi$ is an isomorphism 
\[
\Xi: H^{r}(\Lambda, A) \oplus H^{r+1}(\Lambda, A) \cong H^{r+1}(\Lambda \times_\beta \ZZ, A). 
\]
\end{cor}


%
%

\subsection{Twisted $k$-graph $C^*$-algebras}

\begin{dfn}
\label{def:twisted CK}
Let $\Lambda$ be a row-finite  $k$-graph with no sources and fix $\varphi \in \Zcub2(\Lambda,\TT)$. A Cuntz-Krieger $\varphi$-representation of $\Lambda$ in a $C^*$-algebra $A$ is a set $\{p_v : v \in
\Lambda^0\} \subseteq A$ of mutually orthogonal projections and a set $\{s_\lambda : \lambda \in Q_1 ( \Lambda ) \} \subseteq A$ satisfying
\begin{itemize}
    \item[(TG1)] for al $1\le i\le k$ andl $\lambda \in \Lambda^{e_i}$, $s_\lambda^*s_\lambda =
        p_{s(\lambda)}$;
    \item[(TG2)] for all $1 \le i < j \le k$ and $\mu,\mu' \in \Lambda^{e_i}$,
        $\nu,\nu' \in \Lambda^{e_j}$ such that $\mu\nu = \nu'\mu'$,
    \[
        s_{\nu'} s_{\mu'} = \varphi(\mu\nu)s_\mu s_\nu;\text{ and}
    \]
    \item[(TG3)] for all $v \in \Lambda^0$ and all $i = 1, \dots, k$ such that
        $v\Lambda^{e_i} \not= \emptyset$,
    \[
        p_v =  \sum_{\lambda \in v\Lambda^{e_i}} s_\lambda s_\lambda^*.
    \]
\end{itemize}
\end{dfn}

\begin{dfn}\label{def:twisted algebra}
Let $\Lambda$ be a row-finite  $k$-graph with no sources and let $\varphi \in \Zcub2(\Lambda,\TT)$. We define $C^*_\varphi(\Lambda)$ to be the universal $C^*$-algebra generated by a Cuntz-Krieger $\varphi$-representation of~$\Lambda$.
\end{dfn}

\begin{rmk}
If $\varphi \in \Zcub2 ( \Lambda , \TT )$ is the trivial cocycle, then $\varphi \in \Bcub2 ( \Lambda , \TT )$ and so by \cite[Proposition 5.3]{KPS4} and \cite[Proposition 5.6]{KPS4} we have $C_\varphi^* ( \Lambda ) \cong C^* ( \Lambda )$. 
\end{rmk}

\begin{rmk}
In the context of twisted $k$-graph $C^*$-algebras, we shall be particularly interested in the following part of the exact sequence for $A=\TT$
\[
\cdots \rightarrow H^1(\Lambda, \TT) \xrightarrow{1-\beta^*} H^1(\Lambda , \TT ) \stackrel{j^*}{\longrightarrow} H^2 ( \Lambda \times_\beta \ZZ , \TT ) \stackrel{i^*}{\longrightarrow} H^2 (\Lambda , \TT )
        \xrightarrow{1-\beta^*} H^2 (\Lambda , \TT ) \rightarrow 
          \cdots
\]
\end{rmk}
\section{Main results} \label{sec:main}
In this section we present our $C^*$-algebraic results. In our main result we 
generalise the isomorphism $C^*(\Lambda\times_\beta\ZZ)\cong 
C^*(\Lambda)\rtimes_{\tilde{\beta}}\ZZ$ from \cite[Theorem~3.4]{FPS2009} (in 
the case $l=1$) to the twisted setting. Note that for $A=\TT$ we use 
multiplicative notation; inverses are given by conjugation, and the identity 
element is $1\in \TT$.

\begin{thm} \label{thm:main}
Let $\Lambda$ be a row-finite  $k$-graph with no sources, 
let  $\beta \in \Aut ( \Lambda )$ and let $\varphi \in Z^2 ( \Lambda \times_\beta\ZZ , \TT )$. Then
\begin{enumerate}[(i)]
\item There is an automorphism $\beta_\varphi$ of $C^*_{i^* ( \varphi )} ( \Lambda )$ such that
\begin{equation} \label{eq:betavarphidef}
\beta_\varphi ( p_v ) = p_{\beta v} \text{ and } \beta_\varphi ( s_e ) = \varphi ( \beta e , 1 ) s_{\beta e} \text{ for all } v \in Q_0 ( \Lambda ) \text{ and } e \in Q_1 ( \Lambda ) .
\end{equation}
\item Let $( F_n )_{n \in \NN}$ be an increasing family of finite subsets of $\Lambda^0$ such that $\cup_{n \in \NN} F_n = \Lambda^0$. The sequence $( \sum_{v \in F_n} s_{(v,1)} )_{n \in \NN}$ converges strictly to a unitary $U \in \Mm C^*_{\varphi} ( \Lambda \times_\beta \ZZ )$ satisfying
\begin{equation} \label{eq:Udef}
U p_{(v,0)} U^* = p_{(\beta v , 0)} \text{ and } U s_{(e,0)} U^* = \varphi(\beta e,1)s_{( \beta e , 0 )} \text{ for all } v \in Q_0 ( \Lambda ) \text{ and } e \in Q_1 ( \Lambda ) .
\end{equation}
\item 
There is a homomorphism  
$\pi: C^*_{i^* ( \varphi )} ( \Lambda ) \to C^*_\varphi ( \Lambda \times_\beta 
\ZZ )$ which forms a covariant pair $(\pi,U)$ whose integrated form
$\pi \times U : C^*_{i^* ( \varphi )}(\Lambda) \rtimes_{\beta_\varphi} \ZZ \to 
C^*_{\varphi} ( \Lambda \times_\beta \ZZ )$ is an isomorphism.
\end{enumerate}
\end{thm}

\begin{rmk}\label{rem: the cube}
In the proof of this theorem we need to calculate the faces of a cube 
$(\beta\lambda,1)\in Q_3(\Lambda\times_\beta\ZZ)$, where $\lambda\in 
Q_2(\Lambda)$. Suppose $\lambda=ef=f'e'$, where $e,e' \in \Lambda^{e_i}$ and 
$f,f' \in \Lambda^{e_j}$ such that $1 \le i < j \le k$. We can factorise 
$(\beta\lambda,1)$ according to the following diagram, and then calculate its 
faces.

\vspace{0.2cm}

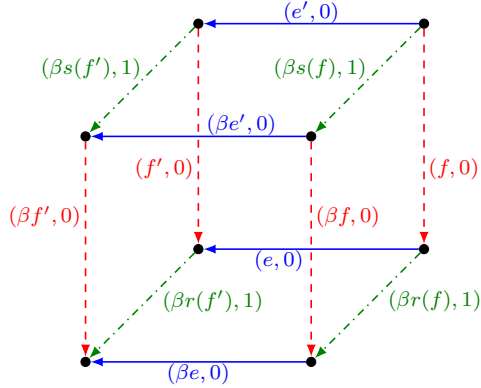
\begin{minipage}[c]{8cm}
\[
\begin{tikzpicture}
    \def\vertex(#1) at (#2,#3){
        \node[inner sep=0pt, circle, fill=black] (#1) at (#2,#3)
        [draw] {.};
    }
    \vertex(11) at (0, -1)
    
    \vertex(21) at (3, -1)
    
    \vertex(31) at (0,2)
    
    \vertex(41) at (3,2)
    
    \vertex(51) at (1.5, 0.5)
    
    \vertex(61) at (4.5, 0.5)
    
    \vertex(71) at (1.5, 3.5)
        
    \vertex(81) at (4.5, 3.5);

%
%
%
%
%
%
%

    
    \draw[style=semithick,color=blue, -latex] (21.west)--(11.east)
    node[pos=0.5,anchor=north,inner sep=0pt] {\tiny$(\beta e,0)$};
    \draw[style=semithick,color=blue, -latex] (41.west)--(31.east)
    node[pos=0.3,anchor=south,inner sep=0pt] {\tiny$(\beta e',0)$};
    \draw[style=semithick,color=blue, -latex] (61.west)--(51.east)
    node[pos=0.65,anchor=north,inner sep=0pt] {\tiny$(e,0)$};
    \draw[style=semithick,color=blue, -latex] (81.west)--(71.east)
    node[pos=0.5,anchor=south,inner sep=0pt] {\tiny$(e',0)$};
    
    \draw[style=semithick,color=red, style=dashed, -latex] 
    (31.south)--(11.north)
    node[pos=0.35,anchor=east,inner sep=0pt] {\tiny$(\beta f',0)\,$};
    \draw[style=semithick,color=red, style=dashed, -latex] 
        (41.south)--(21.north)
        node[pos=0.35,anchor=west,inner sep=0pt] {\tiny$\,(\beta f,0)$};
    \draw[style=semithick,color=red, style=dashed, -latex] 
        (71.south)--(51.north)
        node[pos=0.65,anchor=east,inner sep=0pt] {\tiny$(f',0)\,$};
    \draw[style=semithick,color=red, style=dashed, -latex] 
        (81.south)--(61.north)
        node[pos=0.65,anchor=west,inner sep=0pt] {\tiny$\,(f,0)$};

    \draw[style=semithick,color=green!50!black, style=dashdotted, -latex] 
            (51.south west)--(11.north east)
            node[pos=0.35,anchor=north west,inner sep=0pt] 
                            {\tiny$\,(\beta r(f'),1)$};
    \draw[style=semithick,color=green!50!black, style=dashdotted, -latex] 
                (61.south west)--(21.north east)
                node[pos=0.35,anchor=north west,inner sep=0pt] 
                {\tiny$\,(\beta r(f),1)$};
   \draw[style=semithick,color=green!50!black, style=dashdotted, -latex] 
               (71.south west)--(31.north east)
               node[pos=0.5,anchor=south east,inner sep=0pt] 
                                           {\tiny$\,(\beta s(f'),1)$};
   \draw[style=semithick,color=green!50!black, style=dashdotted, -latex] 
               (81.south west)--(41.north 
               east)
               node[pos=0.5,anchor=south east,inner sep=0pt] 
                                           {\tiny$\,(\beta s(f),1)$};

\end{tikzpicture}
\]
\end{minipage}
\begin{minipage}[c]{8cm}
\begin{align*}
F_1^0(\beta\lambda,1)&=(\beta f',1)\\
F_1^1(\beta\lambda,1)&=(\beta f,1)\\
F_2^0(\beta\lambda,1)&=(\beta e,1)\\
F_2^1(\beta\lambda,1)&=(\beta e',1)\\
F_3^0(\beta\lambda,1)&=(\beta\lambda,0)\\
F_3^1(\beta\lambda,1)&=(\lambda,0)
\end{align*}
\end{minipage}
\end{rmk}

\vspace{0.2cm}

\begin{proof}[Proof of Theorem~\ref{thm:main}]

Let $\lambda\in 
Q_2(\Lambda)$. Write $\lambda=ef=f'e'$, where $e,e' \in \Lambda^{e_i}$ and 
$f,f' \in \Lambda^{e_j}$ such that $1 \le i < j \le k$. Using (\ref{eq: delta 
r of f of lambda one}) and the identities in Remark~\ref{rem: the cube} we get
\begin{equation}\label{eq:cohrel}
1 = \dcub{2}(\varphi)(\beta\lambda,1) =
\overline{\varphi(\beta\lambda,0)}\varphi(\lambda,0)\overline{\varphi(\beta f',1)}\varphi(\beta e,1)
\overline{\varphi(\beta e',1)}\varphi(\beta f,1).
\end{equation} 
For each $v \in Q_0 ( \Lambda )$ let $P_v := p_{\beta v}$ and for each $e \in 
Q_1 ( 
\Lambda )$ let $S_e := \varphi ( \beta e,1) s_{\beta e}$. We claim that $\{ P, 
S \}$ defines a $i^* ( \varphi )$-representation of $\Lambda$ in $C^*_{i^* ( 
\varphi )} ( 
\Lambda )$. We check condition (TG2) using (\ref{eq:cohrel}):
\begin{align*}
S_{f'} S_{e'} &= \varphi ( \beta f' , 1 ) \varphi ( \beta e' , 1 ) s_{\beta f'} s_{\beta e'} \\
&= \varphi ( \beta f' , 1 ) \varphi ( \beta e' , 1 ) i^* ( \varphi ) ( \beta e \beta f) s_{\beta e} s_{\beta f} \\
&= \varphi ( \beta f' , 1 ) \varphi ( \beta e' , 1 )  \varphi ( \beta (ef), 0 ) s_{\beta e} s_{\beta f} \\
&= \varphi ( \beta f' , 1 ) \varphi ( \beta e' , 1 )  \varphi ( \beta (ef), 0 )  \overline{\varphi( \beta e,1)\varphi( \beta f,1)} S_e S_f \\
&= \varphi (ef,0) S_e S_f =  i^* ( \varphi ) ( ef  ) S_e S_f.
\end{align*}
Conditions (TG1) and (TG3) follow easily. The universal 
property of $C_{i^* ( \varphi)}^* ( \Lambda)$ now gives a homomorphism 
$\beta_\varphi$ satisfying $\beta_\varphi(p_v)=p_{\beta v}$ and 
$\beta_\varphi(s_e)=\varphi(\beta e,1)s_{\beta e}$. Similar calculations show 
that the collection $\{p_{\beta^{-1}v},\overline{\varphi(e,1)}s_{\beta^{-1}e}\}$ is also an $i^*(\varphi)$-representation of $\Lambda$ 
in $C_{i^*(\varphi)}^*(\Lambda)$, and the corresponding homomorphism coming 
from the universal property of $C_{i^*(\varphi)}^*(\Lambda)$ is the inverse of 
$\beta_\varphi$. So $\beta_\varphi$ is an automorphism, and (i) holds.

For any finite subset $F\subseteq \Lambda^0$ we denote by $P(F):=\sum_{v\in 
F}p_{(v,0)}\in C^*(\Lambda\times_\beta\ZZ)$. To see that (ii) holds, first let 
$( F_n )_{n \in \NN}$ be an increasing sequence of finite subsets of 
$\Lambda^0$ such that $\cup_{n \in \NN} F_n = \Lambda^0$. Then a standard 
argument shows that $P ( F_n ) \to 1$ strictly in $ \Mm C^*_{\varphi} ( 
\Lambda \times_\beta \ZZ )$. For $n \ge 
1$ let $U_n := \sum_{v \in F_n} s_{(v,1)}$. Since the elements in the sum 
defining $U_n$ have the same degree, by (TG1) we have
\begin{equation} \label{eq:undom}
U_n^* U_n = \sum_{v,w \in F_n} s_{(v,1)}^* s_{(w,1)}= \sum_{v \in F_n} s_{(v,1)}^* s_{(v,1)} 
= \sum_{v \in F_n} p_{(\beta^{-1} v,0)} = P( \beta^{-1} (F_n) ) ,
\end{equation}

\noindent
and by (TG3) we have
\begin{equation} \label{eq:unran}
U_n U_n^* =  \sum_{v,w \in F_n} s_{(v,1)} s_{(w,1)}^*= \sum_{v \in F_n} s_{(v,1)} s_{(v,1)}^* = \sum_{v \in F_n} p_{(v,0)} = P ( F_n ).
\end{equation}

\noindent  Hence $U_n$ is a partial isometry, with initial projection $P( 
\beta^{-1} (F_n))$ and final projection $P(F_n)$.

 For $(w,0) \in ( \Lambda \times_\beta \ZZ )^0$ and $(e,0) \in ( \Lambda \times_\beta \ZZ )^{e_{j}}$, $1 \le j \le k$ we have
\begin{equation} \label{eq:unleft}
\begin{array}{rl}
U_n p_{(w,0)} &= s_{(\beta w,1)} \text{ if } \beta w \in F_n , \text{ and  
zero otherwise, and} \\
U_n s_{(e,0)} &= s_{( \beta r (  e) , 1 )} s_{(e,0)} \text{ if } \beta r(e) 
\in F_n ,  \text{ and  zero otherwise;} 
\end{array}
\end{equation}

\noindent and
\begin{equation} \label{eq:unright}
\begin{array}{rl}
p_{(w,0)} U_n &= s_{(w,1)} \text{ if }  w \in F_n \text{ and  zero otherwise, 
and} \\ 
s_{(e,0)} U_n &= s_{(e,0)} s_{( s ( e) , 1 )} \text{ if }  s(e) \in F_n ,  \text{ and  zero otherwise.}
\end{array}
\end{equation}

\noindent Hence $U_n$ multiplied on the left or right of any product of 
generators of $C_\varphi^* ( \Lambda \times_\beta \ZZ )$ is eventually 
constant as $n \to \infty$. A standard argument shows that $U_n$ converges 
strictly to an element $U\in \Mm C^*(\Lambda\times_\beta\ZZ)$. Moreover, $U$ 
is independent of the choice of $F_n$, and from \eqref{eq:unran} and 
\eqref{eq:undom} we see that $UU^*=1=UU^*$. Finally from \eqref{eq:unleft} and 
\eqref{eq:unright} it follows that for $w , \beta w \in F_n$ we have $U_n 
p_{(w,0)} = p_{(\beta w , 0 )} U_n$, and for $\beta r(e)  , \beta s(e) \in 
F_n$ we have
\[
U_n s_{(e,0)} = s_{(e,0)} s_{(\beta r(e),1)} = \varphi ( \beta e , 1 ) s_{(\beta e , 0)} s_{(\beta s(e) , 1 )} = \varphi ( \beta e , 1 ) s_{(\beta e ,0)} U_n
\]
We see that the identities \eqref{eq:Udef} hold by taking $n \to \infty$. This 
completes the proof of (ii).

For (iii) we first claim that $\{p_{(v,0)},s_{(e,0)}\}$ is an $i^* ( \varphi 
)$-representation of $\Lambda$ in $C^*_\varphi ( \Lambda \times_\beta \ZZ )$.
We check (TG2): for $e,e' \in \Lambda^{e_i}$ and $f,f' \in \Lambda^{e_j}$ such 
that $ef=f'e'$, where $1 \le i < j \le k$, we have
\[
s_{(f',0)} s_{(e',0)} = \varphi ( ef, 0 ) s_{(e,0)} s_{(f,0)}.
\]

\noindent 
Checking conditions (TG1) and (TG3) is straightforward. The universal property 
of $C_{i^* ( \varphi)}^* ( \Lambda)$ now gives a homomorphism $\pi : C^*_{i^* 
( \varphi )} ( \Lambda ) \to C^*_\varphi ( \Lambda \times_\beta \ZZ )$ 
satisfying $\pi ( p_v ) = p_{(v,0)}$ and $\pi ( s_e ) = s_{(e,0)}$.

The homomorphism $\pi$ and the unitary $U$ from (ii) satisfy
\[
U \pi ( s_e ) U^* = U s_{(e,0)} U^* = \varphi ( \beta e,1) S'_{\beta e} =  
\pi \big( \varphi ( \beta e,1) s_{\beta e} \big)= \pi \big( \beta_\varphi ( 
s_{e} ) \big),
\]
for each $e\in Q_1(\Lambda)$. It follows that $(\pi,U)$ is a covariant 
representation of $(C_{i^*(\varphi)}^*(\Lambda),\beta_\varphi)$, and hence by 
the universal property of the full crossed product 
$C_{i^*(\varphi)}^*(\Lambda)\rtimes_{\beta_\varphi}\ZZ$ we get a homomorphism 
$\pi\times U:C_{i^*(\varphi)}^*(\Lambda)\rtimes_{\beta_\varphi}\ZZ\to 
C_\varphi^*(\Lambda\times_\beta\ZZ)$. If we denote the universal covariant 
pair by $(i_\Lambda,i_{\ZZ}(1))$, then we know that $(\pi\times U)\circ 
i_\Lambda=\pi$ and $\overline{\pi\times U}(i_\ZZ(1))=U$, where 
$\overline{\pi\times U}$ is the extension of $\pi\times U$ to the multiplier 
algebra $\Mm C^*_{\varphi} ( \Lambda \times_\beta \ZZ )$.

We claim that $\pi\times U$ is an isomorphism. To find the inverse we 
construct a $\varphi$-representation of $\Lambda\times_\beta\ZZ$ in 
$C_{i^*(\varphi)}^*(\Lambda)\rtimes_{\beta_\varphi}\ZZ$. For each $(v,0) \in 
Q_0 ( \Lambda \times_\beta \ZZ )$ let $P_{(v,0)} := i_{\Lambda } ( p_v )$, for 
each $(e,0) \in Q_1 ( \Lambda \times_\beta \ZZ )$ let $S_{(e,0)} := 
i_{\Lambda} ( 
s_{e} )$, and for each $(v,1) \in Q_1 ( \Lambda \times_\beta \ZZ )$ let 
$S_{(v,1)} := i_\Lambda ( p_{v} ) i_\ZZ ( 1 )$. We claim that $\{ P, S 
\}$ is a Cuntz-Krieger $\varphi$-representation of $\Lambda \times_\beta \ZZ$ 
in $C_{i^* (\varphi )}^* ( \Lambda ) \rtimes_{\beta_\varphi} \ZZ$. To check 
(TG2) we have two cases to consider. For $(\beta e , 1 ) = ( \beta r(e), 1 ) ( 
e,0) = ( \beta e , 0 ) ( \beta s(e) , 1 ) \in ( \Lambda \times_\beta 
\ZZ)^{e_i+e_{k+1}}$ we have
\begin{align*}
S_{(\beta r(e),1)} S_{(e ,0)} &= i_\Lambda ( p_{\beta r(e)} ) i_\ZZ (1)  i_\Lambda ( s_{e } ) \\
&= i_\Lambda ( p_{\beta r(e)} ) i_\Lambda ( \beta_\varphi ( s_{e } ) ) i_\ZZ (1) \\
&= \varphi ( \beta e , 1) i_\Lambda ( p_{\beta r(e)} ) i_\Lambda ( s_{ \beta e 
} ) ) i_\ZZ (1)\\
&= \varphi ( \beta e , 1 ) i_\Lambda ( s_{ \beta e } ) i_\Lambda ( p_{\beta 
s(e)} ) i_\ZZ (1)\\
&= \varphi ( \beta e ,1 ) S_{(\beta e , 0)} S_{(\beta s(e), 1)}  .
\end{align*}
The other case is when $ef=f'e'$, where $e , e'  \in ( \Lambda \times_\beta 
\ZZ)^{e_i}$ and $f, f'  \in ( \Lambda \times_\beta \ZZ )^{e_j}$ and $1 \le i < 
j \le k$. Then $(ef,0)=(e,0)(f,0) = (f',0)(e',0) = (f'e',0)  \in \Lambda 
\times_\beta \ZZ$, and
\[
S_{(f',0)} S_{(e',0)} = i_\Lambda ( s_{f'} s_{e'} ) = i^* ( \varphi ) (ef ) i_\Lambda ( s_e s_f )   = \varphi (ef, 0 ) S_{(e,0)} S_{(f,0)} .
\]
Properties (TG1) and (TG3) follow more easily. The universal property of 
$C_{\varphi}^*( \Lambda \times_\beta \ZZ )$ now gives a homomorphism
$\rho_{P,S} : C_{\varphi}^*( \Lambda \times_\beta \ZZ ) \to C^*_{i^* ( \varphi )} ( \Lambda ) \rtimes_{\beta_\varphi} \ZZ$ such that
$\rho_{P,S} ( p_{(v,0)} ) = P_{(v,0)}$,  $\rho_{P,S} ( s_{(e,0)} ) = S_{(e,0)}$, and $\rho_{P,S} ( s_{(v,1)} ) = S_{(v,1)}$.
One checks on generators that $\rho_{P,S}$ is the inverse of $\pi \times U$.
\end{proof}

\begin{cor} \label{cor:qfa}
Let $\Lambda$ be a row-finite  $k$-graph with no sources, let
$\beta\in\Aut(\Lambda)$ and let $c\in Z^1(\Lambda,\TT)$. There is an automorphism 
$\beta^c$ of $\Lambda$ satisfying
\begin{equation}\label{eq: beta c}
\beta^c ( p_v ) = p_{\beta v} \text{ and } \beta^c ( s_e ) = c ( \beta e  ) 
s_{\beta e}, \text{ for all } v \in Q_0 ( \Lambda ) \text{ and } e \in Q_1 ( 
\Lambda ),
\end{equation}
and an isomorphism $C^*(\Lambda)\rtimes_{\beta^c}\ZZ\cong 
C_{j^*(c)}^*(\Lambda\times_\beta\ZZ)$.\end{cor}

\begin{proof} 
We apply Theorem~\ref{thm:main} with $\varphi:= j^*(c)\in Z^2(\Lambda\times_\beta\ZZ,\TT)$. 
Then $\beta^c$ is just $\beta_{j^*(c)}$, and (\ref{eq: beta c}) follows because $\varphi(e,1) 
= j^*(c)(e,1) = c(e)$ for all $e \in Q_1 ( \Lambda )$. The isomorphism 
$C^*(\Lambda)\rtimes_{\beta^c}\ZZ\cong C_{j^*(c)}^*(\Lambda\times_\beta\ZZ)$ 
follows by Theorem~\ref{thm:main} and realising that 
$i^*(\varphi)=i^*(j^*(c))=1$.
\end{proof}

\begin{cor} \label{cor:2}
Let $\Lambda$ be a row-finite $k$-graph with no sorces and let $\beta \in \Aut ( \Lambda )$. 
Suppose that $\psi \in  Z^2(\Lambda,\TT)$ such that $[\psi] \in \ker ( 1 - \beta^*)$.
Then there is $\varphi \in  Z^2(\Lambda \times_\beta \ZZ,\TT)$ such that $\psi = i^* ( \varphi )$
and so Theorem~\ref{thm:main} applies.
In particular there is an automorphism $\beta_\varphi$ of $C^*_{\psi} ( \Lambda )$ such that
\[
C^*_{i^* ( \varphi )}(\Lambda) \rtimes_{\beta_\varphi} \ZZ \cong 
C^*_{\varphi} ( \Lambda \times_\beta \ZZ ).
\]
\end{cor}

\begin{proof}
Since $[\psi] \in \ker ( 1 - \beta^*)$,  $[\psi] = [\beta^*\psi]$ and so there is a map 
$b: Q_1 ( \Lambda ) \to \TT$ such that $(\beta^*\psi)\overline\psi = \dcub{1}(b)$.
So for $\lambda=ef=f'e'$ where $e,e' \in \Lambda^{e_i}$ and $f,f' \in \Lambda^{e_j}$ 
such that   $1 \le i < j \le k$, we get
\[
\psi(\beta\lambda)\overline{\psi(\lambda)} = b(e)b(f)\overline{b(f')b(e')}.
\]
We define $\varphi \in  Z^2(\Lambda \times_\beta \ZZ ,\TT)$ by
\[
\varphi(\lambda, \ell) = 
\begin{cases}
\psi(\lambda) & \text{if } \lambda \in Q_2(\Lambda), \ell = 0, \\
b(\beta^{-1}\lambda) & \text{if } \lambda \in Q_1(\Lambda), \ell = 1.
\end{cases}
\]
Then by a computation as in Equation (\ref{eq:cohrel}) we have $\dcub{2}(\varphi)(\beta\lambda,1) = 1$
for all $\lambda \in Q_2(\Lambda)$;  moreover,  for $\lambda \in Q_3(\Lambda)$, we have
$\dcub{2}(\varphi)(\lambda,0) = (\dcub{2}(\psi)(\lambda),0) = 1$ since $\psi \in  Z^2(\Lambda,\TT)$ .  
Hence, $\varphi \in  Z^2(\Lambda \times_\beta \ZZ,\TT)$.  By construction $\psi = i^* ( \varphi )$ and
so Theorem~\ref{thm:main} gives the result.
\end{proof}

\noindent Recall that if $\beta$ is the identity automorphism of $\Lambda$ then $\Lambda \times_\beta \ZZ \cong \Lambda \times T_1$ by Remark~\ref{rmk:betatriv} and $\Xi: Z^1(\Lambda, \TT)\oplus Z^2(\Lambda,\TT)\cong Z^2(\Lambda\times T_1,\TT)$ by equation (\ref{eq: Xi}). In this case Theorem~\ref{thm:main} reduces to the following result:

\begin{cor}\label{cor:3}
Let $\Lambda$ be a row-finite $k$-graph with no sources and
let $\varphi \in Z^2(\Lambda\times T_1,\TT)$.  Then $\varphi = \Xi(\varphi_1, \varphi_2)$  where
$(\varphi_1, \varphi_2) \in Z^{1}(\Lambda, \TT) \oplus Z^{2}(\Lambda, \TT)$.  Moreover,
\begin{enumerate}[(i)]
\item
We have $\varphi(e,1) = \varphi_1(e)$ for all $e\in Q_1(\Lambda)$ and 
$\varphi(\lambda, 0) = \varphi_2(\lambda)$ for all $\lambda \in Q_2(\Lambda)$.
\item There is an automorphism $\alpha=\alpha_\varphi$ of $C^*_{\varphi_2}(\Lambda)$ such that
\[
\alpha(p_v)=p_v \text{ and }  \alpha(s_e)=\varphi_1(e)s_e \text{ for all } v\in Q_0(\Lambda) \text{ and } 
e\in Q_1(\Lambda).
\]
\item There is an isomorphism $C^*_{\varphi_2}(\Lambda) \rtimes_\alpha\ZZ\to C_\varphi^*(\Lambda\times T_1)$.
\end{enumerate}
\end{cor}
\section{Examples} 
We consider some examples of automorphisms and the associated crossed products.  
In Example~\ref{ex:1} we consider quasifree automorphisms on Cuntz algebras. 
In Examples~\ref{ex:2} and \ref{ex:3} we compute cohomology of the crossed product; in Example~\ref{ex:3} we adduce conditions under which the twisted crossed product $C^*$-algebra is simple and purely infinite and use classification results to show that it is isomorphic to $\mathcal{O}_2$.

\begin{example}\label{ex:1}
For $n > 1$ let $B_n$ denote the $1$-graph which is the path category of the directed graph with a single vertex $v$ and edges $f_1 , \ldots , f_n$. It is well-known that $C^* ( B_n ) \cong \mathcal{O}_n$. It is straightforward to see that $Z^1 ( B_n , \TT ) \cong \TT^n$, since we may label each edge $f_i$ with an independent element of $\TT$. It is also straightforward to see that $\Aut B_n$ is isomorphic to $S_n$, the symmetric group of order $n$, which acts by permuting the edges $f_i$.

Following \cite{E}, an  automorphism $\alpha$ of $\mathcal{O}_n$ is said to be \emph{quasifree} if it is determined by a unitary matrix $u \in U(n)$ in the following sense
\[
\alpha(s_{f_i}) = \sum_{j=1}^n u_{i,j} s_{f_j}
\text{ for } i =1 , \ldots , n .
\]
We write $\alpha = \alpha_u$.  
Given $u , u' \in U(n)$, we have $\alpha_{uu'} = \alpha_{u}\circ \alpha_{u'}$.
Moreover, if  $u , u' \in U(n)$ are conjugate, the corresponding  automorphisms
 $\alpha_{u}$, $\alpha_{u'}$ are conjugate. 

Evans notes on \cite[Page~917]{E} (citing an argument of Archbold from \cite{A}) that $\alpha_u$ is outer if and only if $u\neq 1$. Hence by \cite[Lemma 10]{KK} the crossed product $\Oo_n\rtimes_{\alpha_u}\ZZ$ is simple and purely infinite if and only if $u^m\neq 1$ for all $m\neq 0$. By the Pimsner-Voiculescu six-term exact sequence we have $K_i(\Oo_n\rtimes_{\alpha_u}\ZZ)=\ZZ/(n-1)\ZZ$ for $i=0,1$. Hence if $u^m\neq 1$ for all $m\neq 0$, the Kirchberg-Phillips Theorem \cite{Kirchberg,Ph} yields that the isomorphism class of $\Oo_n\rtimes_{\alpha_u}\ZZ$ is independent of $u$.

We consider the situation covered in Corollary~\ref{cor:qfa} in the case $\Lambda = B_n$.
If the action $\beta$ on $B_n$ is induced by the identity permutation and 
$c = ( c_1 , \ldots , c_n ) \in Z^1 ( B_n , \TT )$,
then the automorphism $\beta^c$ of Corollary~\ref{cor:qfa} is the quasifree automorphism $\alpha_u$ of 
$\mathcal{O}_n$ arising from the $n \times n$ diagonal matrix $u$ with entries determined by $c$ 
(see  \cite{K,E,Ki}). By Remark~\ref{rmk:betatriv} we have that $B_n \times_{\beta} \ZZ \cong B_n \times T_1$ and so by Corollary~\ref{cor:qfa} we have
\[
\mathcal{O}_n \rtimes_{\beta^c} \ZZ \cong C^*_{j^*(c)} ( B_n \times T_1 ).
\]
Moreover, $u^m\neq 1$ for all $m\neq 0$ if and only if $c_i$ is not a root of unity for some $1\le i\le n$, and hence in this case by the above paragraph we have $C^*_{j^*(c)} ( B_n \times T_1 )$ simple and purely infinite with $K_i(C^*_{j^*(c)} ( B_n \times T_1 ))=\ZZ/(n-1)\ZZ$ for $i=0,1$.

Now let $\alpha$ be a quasifree automorphism of $\mathcal{O}_n$; then $\alpha = \alpha_u$ 
for some unitary matrix $u \in U(n)$.  
Then since every unitary is conjugate to a diagonal unitary, $\alpha$ is conjugate to $\beta^c$ for some
$c \in Z^1 ( B_n , \TT )$ (where $\beta$ is the identity permutation).  
Hence, 
\[
\mathcal{O}_n \rtimes_{\alpha} \ZZ \cong \mathcal{O}_n \rtimes_{\beta^c} 
\ZZ \cong C^*_{j^*(c)} ( B_n \times T_1 ).
\]
Let $\gamma$ be an arbitrary permutation of the edges of $B_n$.  
Then the induced automorphism $\gamma$ of $\mathcal{O}_n$ is the quasifree automorphism associated to the unitary corresponding to the permutation.  
So if $c \in Z^1 ( B_n , \TT )$, then $\gamma^c = \beta^c \circ \gamma$ is also a quasifree automorphism.
%
%
\end{example}

\begin{example}  \label{ex:2}
Consider the infinite $1$-graph $\Lambda$ with vertices $v_n$ for $n\ge 1$ and edges $e_{nj}$, $n\ge 1, j=1,...,n$, where $s(e_{nj})=v_{n+1}$ and $r(e_{nj})=v_n$ for $j=1,...,n$. 
Then $C^*(\Lambda)$ is strongly Morita equivalent to the UHF-algebra with $K_0$-group isomorphic to $\QQ$. 
Let $\beta$ be the automorphism of $\Lambda$ which fixes the vertices and cyclically permutes the edges between two successive vertices, i.e. 
\[
\beta(e_{nj})=e_{n,j+1}, j=1,...,n-1, \beta(e_{nn})=e_{n1}.
\] 
The crossed product graph $\Lambda\times_\beta\ZZ$ has degree $(1,0)$ edges $e_{nj}'=(e_{nj},0)$ for $n\ge 1, j=1,...,n$ and it  has a degree $(0,1)$ loop $f_n=(v_n,1)$ at each vertex $v_n$ 
(we identify $(v_n, 0)$ with $v_n$), with commuting squares
\begin{equation} \label{eq:fp}
e_{nj}'f_{n+1}=f_ne_{n,j+1}', \text{ for }j=1,...,n-1, \quad e_{nn}'f_{n+1}=f_ne_{n1}'.
\end{equation}
It is a rank-$2$ Bratteli diagram (see \cite[Definition 4.1]{PRRS}) with $1$-skeleton shown below
\[
\begin{tikzpicture}
    \def\vertex(#1) at (#2,#3){
        \node[inner sep=0pt, circle, fill=black] (#1) at (#2,#3)
        [draw] {.};
    }
    \vertex(11) at (0, 1)
    
    \vertex(21) at (2, 1)
    
    \vertex(31) at (4, 1)
    
    \vertex(41) at (6, 1)
    
    \node at (7, 1) {$\dots$};

\node at (0,0.7) {$\scriptstyle v_1$};
\node at (2,0.7) {$\scriptstyle v_2$};
\node at (4,0.7) {$\scriptstyle v_3$};
\node at (6,0.7) {$\scriptstyle v_4$};

\node at (0,2) {$\scriptstyle f_1$};
\node at (2,2) {$\scriptstyle f_2$};
\node at (4,2) {$\scriptstyle f_3$};
\node at (6,2) {$\scriptstyle f_4$};

\node at (3,1.5) {$\scriptstyle e'_{21}$};
\node at (3,0.5) {$\scriptstyle e'_{22}$};

\node at (5,1.7) {$\scriptstyle e'_{31}$};
\node at (5,0.33) {$\scriptstyle e'_{33}$};

    \draw[style=semithick, -latex, blue] (21.west)--(11.east) 
    node[black, pos=0.5, anchor=south, inner sep=1pt]{$\scriptstyle e'_{11}$};

\draw[<-, style=semithick, -latex, blue] (31.north west)
    parabola[parabola height=0.2cm] (21.north east);

\draw[<-, style=semithick, -latex, blue] (31.south west)
    parabola[parabola height=-0.2cm] (21.south east);

    \draw[style=semithick, -latex, blue] (41.west)--(31.east) 
    node[black, pos=0.5, anchor=south, inner sep=1pt]{$\scriptstyle e'_{32}$};

\draw[<-, style=semithick, -latex, blue] (41.north west)
    parabola[parabola height=0.35cm] (31.north east);

\draw[<-, style=semithick, -latex, blue] (41.south west)
    parabola[parabola height=-0.35cm] (31.south east);
    \draw[style=semithick, style=dashed, -latex, red] (11.north east)
        .. controls (0.25,1.25) and (0.5,1.75) ..
        (0,1.75)
        .. controls (-0.5,1.75) and (-0.25,1.25) ..
        (11.north west);
    \draw[style=semithick, style=dashed, -latex, red] (21.north east)
        .. controls (2.25,1.25) and (2.5,1.75) ..
        (2,1.75)
        .. controls (1.5,1.75) and (1.75,1.25) ..
        (21.north west);
    \draw[style=semithick, style=dashed, -latex, red] (31.north east)
        .. controls (4.25,1.25) and (4.5,1.75) ..
        (4,1.75)
        .. controls (3.5,1.75) and (3.75,1.25) ..
        (31.north west);
    \draw[style=semithick, style=dashed, -latex, red] (41.north east)
        .. controls (6.25,1.25) and (6.5,1.75) ..
        (6,1.75)
        .. controls (5.5,1.75) and (5.75,1.25) ..
        (41.north west);
\end{tikzpicture}
\]
\noindent
By \cite[Corollary 3.12, Theorem 4.3, Lemma 4.8]{PRRS} the $C^*$-algebra $C^*(\Lambda\times_\beta\ZZ)$ is strongly Morita equivalent to an A$\TT$-algebra, the inductive limit of
\[
C(\TT)\rightarrow C(\TT)\otimes M_{2!}\rightarrow C(\TT)\otimes M_{3!}\rightarrow\cdots
\]

\noindent
with $K$-theory groups isomorphic to $\QQ$.  Fix $m\le n$, then for all $\alpha \in v_n ( \Lambda \times_\beta \ZZ )^{(m,0)}$ we have $v_m ( \Lambda \times_\beta \ZZ ) s(\alpha) \neq \emptyset$ and so $\Lambda \times_\beta \ZZ$ is cofinal. Fix $\alpha \in v_n ( \Lambda \times_\beta\ZZ )^{(N,0)}$, then by repeated use of \eqref{eq:fp} it follows that $o(\alpha)=n (n+1) \cdots (n+N)$ where $o(\alpha)$ is the order of $\alpha$ (see \cite[\S 5]{PRRS}). It is then easy to see that $\Lambda \times_\beta \ZZ$ has large-permutation factorisations (see \cite[Definition 5.6]{PRRS}) and so $C^*(\Lambda\times_\beta\ZZ)$ is simple and has real-rank zero by \cite[Theorem 5.7]{PRRS}.

We now turn our attention to computing  $H^1(\Lambda\times_\beta\ZZ, A)$; first we compute $H^1 ( \Lambda , A)$. Let $A$ be an abelian group. Observe that $\delta^0(f)(\lambda)=f(s(\lambda))-f(r(\lambda))$
where $f \in C^0(\Lambda, A)$, $\lambda\in\Lambda$
and $\delta^1: C^1(\Lambda, A)\to C^2(\Lambda, A)$ is the zero map. 
Hence $H^1(\Lambda, A)=C^1(\Lambda, A)/\text{im}\,\delta^0$ and, if we identify elements 
of $C^1(\Lambda, A)$ with ``double sequences" $a=(a^n_j)_{n\ge 1, 1\le j\le n}$ with $a^n_j\in A$, 
then $H^1(\Lambda, A)$ may be identified with the set of equivalence classes of such elements
where $a\sim b$ if there are $c_n\in A$ with $b^n_j=a^n_j+c_n, n\ge 1, j=1,\dots,n$. It follows that $H^1(\Lambda, A) \cong (\prod_{n\ge 1} A^n)/\!\sim$. 

To compute the cohomology $H^1(\Lambda\times_\beta\ZZ, A)$, observe that 
\[\delta^1(\varphi)(e_{nj}'f_{n+1}=f_ne_{n,j+1}')=\varphi(f_{n+1})+\varphi(e_{nj}')-\varphi(e_{n,j+1}')-\varphi(f_n),\]
so for $\varphi\in Z^1(\Lambda\times_\beta\ZZ,A)$ we have
\[\varphi(e_{n,j+1}')-\varphi(e_{nj}')=\varphi(f_{n+1})-\varphi(f_n)=\varphi(e_{n1}')-\varphi(e_{nn}')\;\text{for}\; n\ge 1, j=1,...,n-1.\]
Summing over $j$ we obtain $n(\varphi(f_{n+1})-\varphi(f_n))=0$ for $n\ge 1$. If $A$ is torsion free, then $\varphi(f_n)$ is constant and therefore $\varphi(e'_{nj})$ is constant. In this case, $H^1(\Lambda\times_\beta\ZZ, A)\cong A$. If $A$ has torsion, then 
\[H^1(\Lambda\times_\beta\ZZ, A)\cong A\times T_2(A)\times T_3(A)\times\cdots,\] where $T_n(A)$ denotes the $n$-torsion subgroup of $A$ for $n\ge 2$.

The last part of the long exact sequence
\[
\cdots \to H^1(\Lambda, A)\stackrel{1-\beta^*}{\longrightarrow}H^1(\Lambda, A)\stackrel{j^*}{\longrightarrow}H^2(\lambda\times_\beta\ZZ,A)\to 0,\]
implies that
\[H^2(\Lambda\times_\beta\ZZ, A)\cong \coker(1-\beta^*).\]
Since $\beta$ cyclically permutes the edges at each stage, the map $\beta^*:C^1(\Lambda, A)\to C^1(\Lambda, A)$ is $\prod \beta_n:\prod_{n\ge 1} A^n\to \prod_{n\ge 1} A^n$, where 
\[\beta_n:A^n\to A^n, \beta_n(a_1,...,a_n)=(a_n, a_1,...,a_{n-1}),\] an automorphism of order $n$. Therefore,
the map $1-\beta^*:C^1(\Lambda, A)\to C^1(\Lambda, A)$ is determined  by 
\[(1-\beta_n)(a_1,...,a_n)=(a_1-a_n, a_2-a_1,...,a_n-a_{n-1}).\]

Since $H^1(\Lambda, A)=C^1(\Lambda, A)/\text{im}\; \delta^0$, we conclude that 
\[
\coker(1-\beta^*)\cong \Bigg(\prod_{n\ge 1} A^n\Bigg)/(\text{im}(1-\beta^*)+\text{im}\; \delta^0).
\]

\noindent
First observe that $(\prod_{n\ge 1} A^n)/\text{im}(1-\beta^*)\cong \prod_{n\ge 1}A$ by the map $(a^n_j)\mapsto (\sum_{j=1}^na^n_j)$. Since $(a^n_j)\in$ im\,$\delta^0$ iff there are $c_n\in A$ with $a^n_j=c_n$ for all $j=1,\dots,n$, we conclude that im\,$\delta^0$ in $\prod_{n\ge 1}A$ is $\prod_{n\ge 1}nA$. It follows that
\[
H^2(\Lambda\times_\beta\ZZ, A)\cong\Bigg(\prod_{n\ge 1}A\Bigg)/\text{im}\,\delta^1\cong \prod_{n\ge 1}A/nA.
\]

\noindent
If $A$ is divisible, in particular if $A = \TT$, then $H^2(\Lambda\times_\beta\ZZ, \TT) \cong 0$.


Given $\varphi\in Z^2(\Lambda\times_\beta\ZZ, \TT)$, both $[\varphi]$ and 
$[i^*(\varphi)]$ are trivial since both $H^2(\Lambda\times_\beta\ZZ, \TT)$ and  $H^2(\Lambda, \TT)$ are trivial. It follows by Theorem \ref{thm:main} that
\[
C^*(\Lambda)\rtimes_{\beta_\varphi}\ZZ \cong C^*_\varphi(\Lambda\times_\beta\ZZ)
\cong C^*(\Lambda\times_\beta\ZZ) .
\]
\end{example}
\begin{example} \label{ex:3}
For $n \ge 1$ let $\underline{n} = \{ 0 , \ldots , n-1 \}$. Define a bijection $\theta : \underline{2} \times \underline{3} \to \underline{2} \times \underline{3}$ by $\theta ( i , j ) = ( i+1 \pmod 2 , j+1 \pmod 3 )$. Consider the $2$-graph $\FF_\theta^+$ defined in \cite{DPY}: $\FF_\theta^+$ is the unital semigroup generated by $f_0,f_1,g_0,g_1,g_2$ subject to the relations $f_i g_j = g_{j'} f_{i'}$ where $\theta (i,j)=(i',j')$, that is
\begin{equation} \label{eq:rels}
f_0 g_0=g_1 f_1, \ f_1 g_0=g_1 f_0, \ f_0 g_1= g_2 f_1, \ f_1 g_1=g_2 f_0, \ f_0 g_2=g_0 f_1, \ f_1 g_2= g_0 f_0 ,
\end{equation}

\noindent  the degree of $f_0,f_1$ is $e_1$ and the degree of $g_0,g_1,g_2$ is $e_2$.

If $\beta_1 = (01) \in S_2$ and $\beta_2 = (012) \in S_3$, then it is easy to check that $\theta \circ ( \beta_1 \times \beta_2 ) = ( \beta_1 \times \beta_2 ) \circ \theta$ and so $\beta = \beta_1 \times \beta_2$ induces an automorphism of $\FF_\theta^+$.

Since $\FF_\theta^+$ has only one vertex $v$, it follows that $v \FF_\theta^+ s( \alpha ) \neq \emptyset$ for all $\alpha \in \FF_\theta^+$ and so $\FF_\theta^+$ is cofinal. Furthermore $\FF_\theta^+$ is aperiodic by \cite[Corollary 3.2]{DY} since $\log 3$ and $\log 2$ are rationally independent. Hence $C^* (\FF_\theta^+ )$ is simple by \cite[Theorem 3.1]{RoSi}. Moreover, the loop $f_0$ has an entrance $f_1 g_0$ and since there is only one vertex, it follows by \cite[Proposition 8.8]{Si} that $C^* ( \FF_\theta^+ )$ is purely infinite. By \cite[Proposition 3.16]{GE} it follows that $K_0 ( C^* ( \FF_\theta^+ )) = K_1 ( C^* ( \FF_\theta^+ )) = 0$ since $M_1 = [2]$ and $M_2 = [3]$.  Since $C^*(\FF_\theta^+)$ is a Kirchberg algebra, $C^*(\FF_\theta^+) \cong \mathcal{O}_2$ by the Kirchberg-Phillips Theorem \cite{Kirchberg,Ph}.

To compute the cohomology of $\Lambda=\FF_\theta^+$ and of $\Lambda\times_\beta\ZZ$, we first compute the homology $H_n(\Lambda)$.  Next we determine  the maps $\beta_*:H_n(\Lambda)\to H_n(\Lambda)$ induced by the automorphism $\beta$ and then use the exact sequence  (see \cite[Theorem 4.13]{KPS3})  to compute $H_n(\Lambda\times_\beta\ZZ)$. Thereafter, we apply the Universal Coefficient Theorem  (see \cite[Theorem 7.3]{KPS3}) 
\[
0\to \Ext(H_{n-1}(\Lambda\times_\beta\ZZ),A)\to H^n(\Lambda\times_\beta\ZZ, A)\to \Hom(H_n(\Lambda\times_\beta\ZZ), A)\to 0
\]
to compute  $H^n(\Lambda\times_\beta\ZZ, A)$. 
We have 
\[
Q_0(\Lambda)=\{v\}, \qquad Q_1(\Lambda)=\{f_0,f_1,g_0,g_1,g_2\},
\]
\[ 
Q_2(\Lambda)=\{f_0 g_0,  \ f_0 g_1, \  f_0 g_2, \ f_1 g_0, \ f_1 g_1, \ f_1 g_2\} ,
\] 
\noindent from \eqref{eq:rels}.
Moreover, we have $\partial_0 = \partial_1=0$ and
\begin{align*}
&&\partial_2(f_0g_0) &= f_0+g_0-g_1-f_1, & \partial_2(f_1g_0)&=f_1+g_0-g_1-f_0, && \\
&& \partial_2(f_0g_1) &=f_0+g_1-g_2-f_1, & \partial_2(f_1g_1)&=f_1+g_1-g_2-f_0, &&\\
&& \partial_2(f_0g_2) &=f_0+g_2-g_0-f_1, & \partial_2(f_1g_2)&=f_1+g_2-g_0-f_0.
\end{align*}
Using the Smith normal form of $\partial_2$ we see that $\ker\partial_2$ has generators
\begin{align*}
b_1 &= f_0g_0+f_1g_0+2f_1g_1+2f_0g_2, \\
b_2 &= -2f_0g_0+f_0g_1-3f_1g_1-2f_0g_2, \\
b_3 &=2f_0g_0+2f_1g_1+f_0g_2+f_1g_2
\end{align*}
and $\text{im }\partial_2$ has generators 
\[f_0-f_1+g_0-g_1,\quad g_0-g_2, \quad g_1-g_2.\]
It follows that
\[H_0(\Lambda)=\ZZ,\quad  H_1(\Lambda)=\ZZ^5/\text{im }\partial_2\cong\ZZ^2,\quad  H_2(\Lambda)=\ker\partial_2\cong\ZZ^3.\]

Since $\beta(f_0)=f_1, \beta(f_1)=f_0, \beta(g_0)=g_1, \beta(g_1)=g_2, \beta(g_2)=g_0$ and $H_1(\Lambda)$ is generated by $[f_0], [g_0]$, we see that $\beta_*:H_1(\Lambda)\to H_1(\Lambda)$ is the identity map.

Since $\beta(b_1)=2b_1+b_2, \beta(b_2)=-2b_1+b_3, \beta(b_3)=b_1$, it follows that $\beta_*:H_2(\Lambda)\to H_2(\Lambda)$ has matrix
\[\left[\begin{array}{rrr}2&1&0\\-2&0&1\\1&0&0\end{array}\right]\]
and $\ker(1-\beta_*)\cong \ZZ, \text{im}(1-\beta_*)\cong \ZZ^2$.

Then by the long exact sequence of homology (see \cite[Theorem 4.13]{KPS3}) we obtain
\[H_1(\Lambda\times_\beta\ZZ)\cong \ZZ^3, \quad H_2(\Lambda\times_\beta\ZZ)\cong \ZZ^3, \quad H_3(\Lambda\times_\beta\ZZ)\cong \ZZ.
 \]
It follows that 
\[
H^0(\Lambda, A)\cong A, \quad H^1(\Lambda, A)\cong A^2,\quad H^2(\Lambda, A)\cong A^3.
\]
and 
\[
H^1(\Lambda\times_\beta\ZZ,A)\cong A^3, \quad H^2(\Lambda\times_\beta\ZZ, A)\cong A^3,\quad H^3(\Lambda\times_\beta\ZZ,A)\cong A.
\]
\noindent
We construct a representation of $C^*( \FF_\theta^+ )$ on the Hilbert space ${\mathcal H}=\ell^2({\mathbb N})\otimes {\mathbb C}^2\otimes\ell^2({\mathbb N})$ with basis $\{e_{i,j,k}\}_{i,k\ge 1,j=0,1}$ by taking
\[S_{f_0} (e_{i,0,k})=e_{2i-1,0,k}, S_{f_0} (e_{i,1,k})=e_{2i,1,k}, \ \ S_{f_1} (e_{i,0,k})=e_{2i,0,k}, S_{f_1} (e_{i,1,k})=e_{2i-1,1,k},
\]
\[
T_{g_0}(e_{i,j,k})=e_{i,1-j,3k-2}, \ T_{g_1} (e_{i,j,k})=e_{i,1-j,3k-1}, \ T_{g_2} (e_{i,j,k})=e_{i,1-j,3k}.
\]

\noindent
Then $S_{f_0}, S_{f_1}, T_{g_0},T_{g_1},T_{g_2}$ are isometries and $S_{f_0} S_{f_0}^*+S_{f_1}S_{f_1}^*=T_{g_0} T_{g_0}^*+T_{g_1}T_{g_1}^*+T_{g_2}T_{g_2}^*=I$. The commutation relations \eqref{eq:rels} are also satisfied, for example
\[
(S_{f_0}T_{g_2})(e_{i,0,k})=S_{f_0} (e_{i,1,3k})=e_{2i,1,3k}, \ (S_{f_0} T_{g_2})(e_{i,1,k})=S_{f_0} (e_{i,0,3k})=e_{2i-1,0,3k} .
\]

Fix $z, w \in \TT$ and let $c = c(z,w) \in Z^1 ( \FF_\theta^+ , \TT )$ be such that $c(f_0)=c(f_1)=z, c(g_i)=w$, for $i=0,1,2$. Suppose that  $z^n , w^n \neq 1$ for all $n$. 
As in Corollary~\ref{cor:qfa}  let $\kappa= \beta^c$  be the automorphism of $C^* ( \FF_\theta^+ )$ such that
\[
\kappa ( S_{f_i} ) =  z S_{f_{\beta_1 (i)} } 
\text{ for } i =0,1 
\text{ and }
\kappa ( T_{g_i} ) = w T_{g_{\beta_2 (i)} } \text{ for } i = 0,1,2 .
\]

\noindent 
We claim that $\kappa^n$ is outer for all $n$. To prove this we follow the technique of \cite{A,ETW1,ETW2}.

\noindent
Fix $n$. For contradiction assume there is  a unitary $V \in C^*( \FF_\theta^+ )$ such that $\kappa^n = Ad V$. We have $S_{f_0} e_{1,0,1}=e_{1,0,1}$ and 
\[
\kappa^n ( S_{f_0} ) V e_{1,0,1}=V S_{f_0} V^*V e_{1,0,1} = V e_{1,0,1} = \sum c_{i,j,k}e_{i,j,k} \neq 0.
\]

\noindent
On the other hand if $n$ is odd then
\[
\kappa^n ( S_{f_0} ) V e_{1,0,1}= z^n S_{f_1} \sum c_{i,j,k}e_{i,j,k} = z^n  \sum_{i,k=1}^\infty c_{i,0,k} e_{2i,0,k} + z^n \sum_{i,k=1}^\infty c_{i,1,k}e_{2i-1,1,k}.
\]

\noindent
Identifying coefficients, we get $c_{2i-1,0,k}=c_{2i,1,k}=0, c_{2i,0,k} = z^n c_{i,0,k}, c_{2i-1,1,k} = z^n c_{i,1,k}$, since $z^n \neq 1$ for all $n$ it follows by induction that $c_{i,j,k}=0$ for all $i,j,k$. If $n$ is even a similar argument applies.

Similarly, consider $x=\frac{1}{2}(e_{1,0,1}+e_{1,1,1})$. Then $T_{g_0} x = x$ and
\[
\kappa^n (T_{g_0})V x = V T_{g_0} V^*V x = Vx = \sum_{i,k=1}^\infty x_{i,j,k} e_{i,j,k} \neq 0.
\]

\noindent
On the other hand if $n$ is congruent to $1$ mod $3$ then we have
\[
\beta(T_{g_0}) V x = w^n T_{g_1}\sum_{i,k=1}^\infty x_{i,j,k} e_{i,j,k} = w^n \sum_{i,k=1}^\infty x_{i,0,k} e_{i,1,3k-1} + w^n \sum_{i,k=1}^\infty x_{i,1,k} e_{i,0,3k-1}.
\]

\noindent
Identifying coefficients,
\[
x_{i,0,3k} = x_{i,0,3k-2} = x_{i,1,3k} = x_{i,1,3k-2} = 0, \ x_{i,0,3k-1} = w^n x_{i,1,k}, \ x_{i,1,3k-1} = w^n x_{i,0,k}
\]

\noindent
and $x_{i,0,2} = w^n x_{i,1,1} =0$. Similarly $x_{i,1,2} = w^n x_{i,0,1}=0$ , since $w^n \neq 1$ for all $n$ it follows by induction that $c_{i,j,k}=0$ for all $i,j,k$. Other cases for $n$ follow in a similar manner.
Together they give a contradiction, so $\kappa^n$ is outer for all $n$. Hence by \cite[Lemma 10]{KK} and Corollary~\ref{cor:qfa} it follows that $C^* ( \FF_\theta^+ ) \rtimes_{\beta^c} \ZZ \cong C_{j^*(c)} ( \FF_\theta^+ \times_\beta \ZZ )$ is simple and purely infinite for all $c=c(z,w) \in Z^1 ( \FF_\theta^+ , \TT )$ such that $z^n , w^n \neq 1$ for all $n$. For such $c$ the Pimsner-Voiculescu six-term exact sequence and classification results show that $C^* ( \FF_\theta^+ ) \rtimes_{\beta^c} \ZZ \cong \mathcal{O}_2$.
\end{example}


\begin{thebibliography}{lll}

\bibitem{A} R.J.~Archbold, \textit{On the `Flip-Flop' Automorphism of $C^* ( S_1 , S_2 )$}, {Quart.\ J.\ Math.,} {\bf 30} (1979), 129--132.

\bibitem{DY} K.R.\ Davidson and D.\ Yang, \textit{Periodicity in rank 2 graph algebras}, Canad. J. Math., {\bf 61} (2009),  1239--1261.

\bibitem{DPY} K.R. Davidson, S.C. Power and D. Yang, \textit{Atomic representations of rank 2 graph algebras}, J. Funct. Anal., \textbf{255} (2008), 819--853.

\bibitem{ETW1} M.Enomoto, H. Takehana, Y. Watatani, {\em  Automorphisms on Cuntz algebras}, Math. Japon. {\bf 24} (1979/80), no. 2, 231--234.

\bibitem{ETW2} M.Enomoto, H. Takehana, Y. Watatani, {\em Automorphisms on Cuntz algebras II}, Math. Japon. {\bf 24} (1979/80), no. 4, 463--468.

\bibitem{E} D. Evans, \emph{On $O_n$}, Publ. Res. Inst. Math. Sci. {\bf 16} (1980), 915--927.

\bibitem{GE} D.G. Evans, \emph{On the $K$-theory of higher-rank graph
    {$C^*$}-algebras}, New York J. Math. \textbf{14} (2008), 1--31.

\bibitem{FPS2009} C. Farthing, D. Pask and A. Sims, \emph{Crossed products of $k$-graph
    $C^\ast$-algebras by $\mathbb{Z}^l$}, Houston J. Math. {\bf 35} (2009), 903--933.

\bibitem{Hatcher2002} A. Hatcher, Algebraic topology, Cambridge University Press, Cambridge,
    2002, xii+544.

    
\bibitem{K} T. Katsura, \emph{The Ideal Structures of Crossed Products of
Cuntz Algebras by Quasi-Free Actions of Abelian Groups}, Canad. J. Math. {\bf 55} (2003) 1302--1338.

\bibitem{Kirchberg} E. Kirchberg, {\em The classification of purely infinite $C^*$-algebras
using Kasparov's theory}, Preprint, 1994.


\bibitem{Ki} A. Kishimoto, \emph{Simple crossed products of $C^*$-algebras by locally compact abelian groups}, Yokohama Math. J. {\bf28}(1980), 69--85.

\bibitem{KK} A. Kishimoto and A. Kumjian, \emph{Crossed products of Cuntz algebras by quasi-free automorphisms}, Operator algebras and their applications (Waterloo, ON, 1994/1995), 173–-192, Fields Inst. Commun., \textbf{13}, Amer. Math. Soc., Providence, RI, 1997.
    
\bibitem{KP2000} A. Kumjian  and D. Pask, \emph{Higher rank graph {$C^*$}-algebras}, 
New York J. Math. \textbf{6} (2000), 1--20.


\bibitem{kps2}  A. Kumjian, D. Pask and A. Sims, \emph{Generalised morphisms of $k$-graphs:
    $k$-morphs}, Trans. Amer. Math. Soc.  {\bf 363} (2011) 2599--2626.

\bibitem{KPS3} A. Kumjian, D. Pask and A. Sims, \emph{Homology for higher-rank graphs
    and twisted $C^*$--algebras}, J. Funct. Anal., \textbf{263} (2012), 1539--1574.

\bibitem{KPS4} A. Kumjian, D. Pask and A. Sims, \emph{On twisted higher-rank graph
    $C^*$-algebras}, to appear Trans.  Amer.\ Math.\ Soc.\ {\tt http://arviv.org/pdf/1112.6233v1}.

\bibitem{LS} P. Lewin and A. Sims, \textit{Aperiodicity and cofinality for finitely aligned higher-rank graphs}, Math. Proc. Cambridge Philosophical Soc., \textbf{149} (2010), 333--350.
    
\bibitem{mpr} B. Maloney, D. Pask and I. Raeburn, B.\ Maloney, D.\ Pask and I.\ Raeburn, \textit{Skew products of higher-rank graphs and crossed products by semigroups}, Semigroup Forum \textbf{88} (2014), 162--176.
    
\bibitem{ML} S. Mac Lane, Homology, Grundlehren Math. Wiss., vol. 114, Springer-Verlag, Berlin-Heidelberg-New York, 1967.

\bibitem{PRRS} D. Pask, I. Raeburn, M. R{\o}rdam, and A. Sims, \emph{Rank-two graphs
    whose {$C^*$}-algebras are direct limits of circle algebras}, J. Funct. Anal.
    \textbf{239} (2006), 137--178.

\bibitem{pqr1} D. Pask, I. Raeburn and J. Quigg,
    \emph{ Fundamental groupoids of $k$-graphs}, New York. J. Math.
    {\bf 10} (2004), 195--207.
    
 \bibitem{prw2} D. Pask, I. Raeburn and N. Weaver, \emph{Periodic $2$-graphs arising from subshifts}, Bull. Aust. Math. Soc. {\bf 82} (2010), 120--138.

\bibitem{Ph} N.C. Phillips, \emph{A classification theorem for nuclear purely
    infinite simple {$C\sp *$}-algebras}, Doc. Math. \textbf{5} (2000), 49--114.

\bibitem{RSY2003} I. Raeburn, A. Sims, and T. Yeend, \emph{Higher-rank graphs and their {$C\sp
    *$}-algebras}, Proc. Edinb.
    Math. Soc. (2) \textbf{46} (2003), 99--115.

\bibitem{Renault1980}  J. Renault, A groupoid approach to $C^*$-algebras, Lecture Notes in Math.
    {\bf 793}, Springer-Verlag,  Berlin-New York, 1980.

\bibitem{RoSi} D. I.\ Robertson and A.\ Sims, \textit{Simplicity of $C^*$-algebras associated to higher rank graphs}, Bull. London Math. Soc., \textbf{39} (2007), 337--344.

\bibitem{Si} A. Sims, \emph{Gauge-invariant ideals in the {$C\sp *$}-algebras of
    finitely aligned higher-rank graphs}, Canad. J. Math. \textbf{58} (2006), 1268--1290.

\bibitem{West}  J. Westman, \emph{Cohomology for ergodic groupoids}, Trans. Amer. Math. Soc. {\bf
    146} (1969), 465--471.


\end{thebibliography}
\end{document}